\documentclass[letterpaper]{amsart}
\usepackage{amsmath, amsthm, amsrefs, amssymb, enumitem, mathtools, hyperref, color, comment, pgfplots}
\pgfplotsset{compat=1.18}

\theoremstyle{plain}
\newtheorem{theorem}{Theorem}
\newtheorem{corollary}[theorem]{Corollary}
\newtheorem{lemma}[theorem]{Lemma}
\newtheorem{proposition}[theorem]{Proposition}

\theoremstyle{definition}
\newtheorem{definition}[theorem]{Definition}

\theoremstyle{remark}
\newtheorem{remark}[theorem]{Remark}

\numberwithin{equation}{section}
\numberwithin{theorem}{section}

\DeclareMathOperator{\D}{\mathbb{D}}
\DeclareMathOperator{\C}{\mathbb{C}}

\DeclareMathOperator{\diam}{diam}
\DeclareMathOperator{\dist}{dist}
\newcommand{\N}{\mathbb N}
\newcommand{\R}{\mathbb R}

\newcommand{\br}{\overline}

\begin{document}
\title{Characterization of John domains via weak tangents}

\author{Christina Karafyllia}  
\address{Department of Mathematics, University of Western Macedonia, Kastoria 52100, Greece}
\email{chkarafyllia@uowm.gr}   


\begin{abstract}
We characterize simply connected John domains in the plane with the aid of weak tangents of the boundary. Specifically, we prove that a bounded simply connected domain $D$ is a John domain if and only if, for every weak tangent $Y$ of $\partial D$, every connected component of the complement of $Y$ that ``originates" from $D$ is a John domain, not necessarily with uniform constants. Our main theorem improves a result of Kinneberg \cite{Kin}, who obtains a necessary condition for a John domain in terms of weak tangents but not a sufficient one. We also establish several properties of weak tangents of John domains.
\end{abstract}

\maketitle

\section{Introduction}

The object of this paper is to characterize simply connected planar John domains by studying weak tangents of the boundary. Intuitively, this means that we zoom into the boundary of the domain at certain points, along with a sequence of scales tending to zero. More precisely, let $X,Y$ be two closed sets in the complex plane $\C$. We say that $Y$ is a \textit{weak tangent} of $X$ if there is a sequence of points $p_n\in X$ and scales $\lambda_n>0$, $n\in \N$, with  $\lambda_n\to 0$ as $n\to\infty$ such that if 
\[X_n=\frac{X-p_n}{\lambda_n}\]
then the sequence of sets $\{X_n\}_{n\in\N}$ converges to $Y$. We explain the precise notion of convergence in Section \ref{prelim}. Note that one can also define weak tangents of arbitrary metric spaces but, for our purposes, we restrict to planar sets. Weak tangents play a significant role in the study of fractal metric spaces since they describe the infinitesimal structure of such spaces. 

Often, one can extract information about the global geometry of a space by studying its infinitesimal structure. For example, weak tangents of the Sierpi\'nski carpet have been studied by Bonk and Merenkov \cite{BonkMer} in connection with the problem of finding all quasisymmetric automorphisms of the carpet. Moreover, weak tangents have been studied by Hakobyan and Li \cite{Hak} in connection with the problem of quasisymmetrically embedding a metric space into Euclidean space.

We focus on the problem of characterizing a metric space by studying its weak tangents. A first motivating result is that a Jordan curve in the plane is a quasicircle if and only if its weak tangents are connected \cite[Theorem 5.3]{Kin}. For higher dimensional objects similar results are false in general. For instance, Wu \cite[Theorem 1.0.10]{Wu} shows that for each $n\geq 2$ there is a metric space $X$ homeomorphic to the $n$-sphere whose weak tangents are isometric to $\R^n$, but $X$ is not a quasisphere (see also \cite[Corollary 4.5]{Li}).

Bonk and Kleiner \cite{Bonk} provide a positive result for topological $n$-spheres $X$ that arise as the boundary at infinity of a Gromov hyperbolic finitely generated group $G$. Specifically, a consequence of \cite[Lemma 5.2]{Bonk} is that such a boundary $X$ is a quasisphere if and only if every weak tangent of $X$ is quasisymmetric to $\R^n$. Such a characterization was also obtained by Wu \cite[Theorem 1.0.9]{Wu} for visual spheres of expanding Thurston maps (see also \cite[Theorem 1.5]{Li}).

In \cite[Corollary 3.8]{Vai}, V\"ais\"al\"a studies uniform domains, a class of domains that contains quasicircles, and among other results he gives a characterization of uniform domains in terms of weak tangents. In this paper we study the weak tangents of a more general class of domains known as John domains. John domains were introduced by F.\ John \cite{Jo} and studied by Martio--Sarvas \cite{MS}. Roughly speaking, a planar domain is a John domain if any two points of the domain can be connected by a curve that is not too close to the boundary. In other words, John domains are not allowed to have outward cusps, but they can have inward cusps.  

John domains arise naturally in distortion problems of conformal and quasiconformal mappings, and in complex dynamics. For example, in \cite{CYZ} Carleson, Jones and Yoccoz give a classification of polynomials whose Fatou components are John domains. In \cite{LR}, Lin and Rhode characterize the laminations of the disk that arise from conformal mappings onto planar dendrites whose complement is a John domain.
Simply connected John domains can be considered as one-sided quasidisks. In fact, Nakki--V\"ais\"al\"a \cite[Theorem 9.3]{Nakki} prove that a Jordan curve in the plane is a quasicircle if and only if both connected components of its complement are John domains. 

By their original definition John domains are bounded. However, later, Nakki--V\"ais\"al\"a \cite{Nakki} extended the definition to allow for unbounded domains. In this paper we adopt their definition from \cite{Nakki}: 

\begin{definition}
Let $D\subset \C$ be a domain and let $A\geq 1$.  We say that $D$ is an \textit{$A$-John domain} if for every pair of distinct points $a,b\in D$ there exists a curve $\gamma\colon [0,1]\to D$ such that  $\gamma(0)=a$, $\gamma(1)=b$ and
$$ \min\{\ell (\gamma|_{[0,t]}), \ell(\gamma|_{[t,1]} )\} \leq A \dist(\gamma(t),\partial D),$$
for all $t\in [0,1]$. In this case $\gamma$ is called an \textit{$A$-John curve} connecting $a$ to $b$. We say that $D$ is a John domain if it is an $A$-John domain for some $A\ge 1$. 
\end{definition}

Kinneberg \cite[Proposition 1.5]{Kin} proved that if $D$ is a bounded simply connected John domain, then the weak tangents of $\partial D$ have a uniformly bounded number of connected components. It is easy to see that the converse is not true. For example, one can construct a Jordan region $D$ with an outward cusp (like a water droplet), which is therefore not a John domain, but which has the property that every weak tangent of $\partial D$ has at most two connected components. Note, however, that the complement of this Jordan region is a John domain. One can easily generalize this construction and allow for both inward and outward cusps while keeping the property that the weak tangents have at most two connected components. Thus, there seems to be no characterization of John domains that takes into account only the geometry of weak tangents $Y$ of $\partial D$, but one has to study as well the connected components of $\C\setminus Y$.
 
For this reason, given a domain $D\subset \C$ and a weak tangent $Y$ of $\partial D$, we wish to distinguish between connected components of $\C\setminus Y$ that ``originate" from $D$ and connected components that do not have this property. If $Y$ is a weak tangent of $\partial D$, 
there is a sequence of points $p_n\in \partial D$ and scales $\lambda_n>0$, $n\in\N$, with $\lambda_n\to 0$ as $n\to\infty$ such that
\[\left\{\frac{\partial D-p_n}{\lambda_n}\right\}_{n\in\N}\]
converges to $Y$ in the local Hausdorff sense, which we define in Section \ref{prelim}. For $n\in \N$ we set 
\[D_n=\frac{D-p_n}{\lambda_n}.\]
We say that a connected component $U$ of $\C\setminus Y$ is a \textit{$D$-component} if there exists $z\in U$ such that $z\in D_n$ for infinitely many $n\in \N$. Note that the weak tangent $Y$ and the $D$-components of $\mathbb C\setminus Y$ depend on the sequences $\lambda_n$ and $p_n$. Our main result gives a necessary and sufficient condition for a bounded simply connected domain $D\subset\C$ to be a John domain by studying the geometry of $D$-components.

\begin{theorem}\label{maintheorem}
Let $D\subset \C$ be a  bounded simply connected domain. The following are equivalent.
\begin{enumerate}[label=\normalfont(\arabic*)]
\item $D$ is a John domain.
\item There exists $A\geq 1$ such that for every weak tangent $Y$ of $\partial D$, every $D$-component of $\C \setminus Y$ is an $A$-John domain.
\item For every weak tangent $Y$ of $\partial D$, every $D$-component of $\C \setminus Y$ is a John domain.
\end{enumerate}
\end{theorem}

\begin{remark}
In (3) we do not require that $D$-components are John domains with a uniform constant. Moreover, if $D$ is an $A$-John domain in (1) then we actually prove that (2) holds for the same constant $A$. However, the implication from (2) to (1) is not quantitative. For example, a domain $D$ bounded by a smooth Jordan curve is always a John domain, but not necessarily with a small constant $A$, while the weak tangents of $\partial D$ are actual tangent lines.
\end{remark}

\begin{remark} The theorem is not true for unbounded domains. Actually, the implication from (1) to (3) (or (2)) holds even if $D$ is unbounded but the converse does not. For example, consider the domain 
$$D=\C\setminus ((-\infty,-1] \cup [1,\infty)),$$ 
which is not a John domain. Yet, every weak tangent $Y$ of $\partial D$ is either a line or a ray and the connected components of $\C\setminus Y$ are John domains. 
\end{remark}

The proof of Theorem \ref{maintheorem} is given in Section \ref{mainthproof} and the most involved implication is that (3) implies (1). In Section \ref{propertieswt}, we establish some properties of weak tangents of John domains.

\begin{theorem}\label{sum}
Let $D\subset \C$ be a bounded simply connected John domain and let $Y$ be a weak tangent of $\partial D$. The following statements are true.
\begin{enumerate}[label=\normalfont(\arabic*)]
	\item Every $D$-component of $\C\setminus Y$ is unbounded and simply connected.
	\item Every connected component of $Y$ is unbounded. 
	\item $Y$ is equal to the union of the boundaries of $D$-components of $\C\setminus Y$.
	\item The number of $D$-components of $\C\setminus Y$ is bounded, quantitatively.
	\item The number of connected components of $Y$ is bounded, quantitatively.
\end{enumerate}
\end{theorem}

If $D$ is an $A$-John domain, in (4) and (5) the term ``quantitatively" means that the relevant constant depends only on $A$. Properties (1)--(5) are established in Lemma \ref{dcompounbounded}, Lemma \ref{unboundedY}, Lemma \ref{components}, Lemma \ref{upperbound} and Corollary \ref{corol}, respectively. We remark that (5) was also obtained by Kinneberg \cite[Theorem 5.6]{Kin} through the notion of the discrete UWS property (uniformly well-spread discrete cut points). Here, it follows from Theorem \ref{maintheorem} and properties (1)--(4). We remark that, given a bounded simply connected domain $D\subset \C$ and a weak tangent $Y$ of $\partial D$, none of conditions (1)--(5) imply that $D$ is a John domain. This can be easily seen by considering, for example, a Jordan region with an outward cusp (like a water droplet). Note that all of our proofs rely on planar techniques. An interesting question is whether any of the preceding results remain valid for John domains in $\R^n$.

\section{Preliminaries}\label{prelim}

\subsection{Convergence of metric spaces}

A \textit{pointed metric space} $(Y,d,p)$ is a metric space $(Y,d)$ along with a distinguished base-point $p\in Y$. For $R>0$ and $a\in Y$, we set $B_d(a,R)=\{y\in Y: d(y,a)<R\}$. When it is clear which metric is used, we denote the ball simply by $B(a,R)$.  We give the following definition according to \cite[Section 8]{Bur}.

\begin{definition} A sequence of pointed metric spaces $\{(Y_n,d_n,p_n)\}_{n\in\N}$ converges to a pointed metric space $(Y,d,p)$ in the \textit{(pointed) Gromov--Hausdorff sense} if the following holds: For every $\varepsilon>0$ and $R>0$, there is $N\in \N$ such that for every $n\ge N$, there is a mapping $f_n\colon B_{d_n}(p_n,R)\to Y$ for which
\begin{enumerate}
\item $f_n(p_n)=p$,
\item $|d_n(x,y)-d(f_n(x),f_n(y))|<\varepsilon$ for each pair $x,y\in B_{d_n}(p_n,R)$ and
\item $B_{d}(p,R-\varepsilon)\subset N_{\varepsilon}(f_n(B_{d_n}(p_n,R)))$.
\end{enumerate} 
\end{definition}

Recall that the Hausdorff distance between two sets $E,F$ is defined to be the infimum of $r>0$ such that $E$ is contained in the open $r$-neighborhood $N_r(F)$ of $F$ and $F$ is contained in $N_r(E)$.
For our purposes we will need the following modification of Hausdorff convergence of subsets of a metric space.
\begin{definition}
Let $(Z,d)$ be a metric space. For $n\in \N$ let $Y_n,Y\subset Z$ be closed sets. We say that the sequence $\{Y_n\}_{n\in \N}$ converges in the \textit{local Hausdorff sense} to the set $Y$ if for each point $z_0\in Z$ and for each $R,\varepsilon>0$ there exists $N\in \N$ such that for every $n\geq N$ we have
\begin{align*}
&Y_n\cap B(z_0,R-\varepsilon)\subset N_\varepsilon(Y\cap B(z_0,R))\,\,\, \textrm{and}\,\,\,Y\cap B(z_0,R-\varepsilon)\subset N_\varepsilon(Y_n\cap B(z_0,R)). 
\end{align*}
\end{definition}

\begin{remark}
The definition of local Hausdorff convergence does not require that $Y_n\cap B(z_0,R)$ converges in the Hausdorff sense to $Y\cap B(z_0,R)$ and it is important to consider instead balls of radius $R-\varepsilon$. For example, if $Y_n=\{1/n\}\subset \R$, $n\in \N$, and $Y=\{0\}\subset \R$, then $Y_n\to Y$ in the (local) Hausdorff sense. However, $B(1,1)=(0,2)$ and $Y\cap (0,2)=\emptyset$, so $Y_n\cap (0,2)$ does not converge to $Y\cap (0,2)$ in the Hausdorff sense.
\end{remark}

\begin{lemma}\label{distconvergence}
Let $(Z,d)$ be a metric space. For $n\in \N$ let $Y_n,Y\subset Z$ be closed sets and let $z_n,z\in Z$. Suppose that $\{Y_n\}_{n\in \N}$ converges in the {local Hausdorff sense} to $Y$ and $d(z_n,z)\to 0$ as $n\to \infty$. Then $\dist(z_n,Y_n)\to \dist(z,Y)$ as $n\to \infty$.
\end{lemma}

\begin{proof} 
Let $R=\dist(z,Y)\geq 0$ and $\varepsilon>0$. We have
$$\emptyset\neq Y\cap B(z,R+\varepsilon) \subset N_\varepsilon(Y_n\cap B(z,R+2\varepsilon))\subset N_\varepsilon(Y_n\cap B(z_n,R+3\varepsilon))$$
for all sufficiently large $n\in \N$. This shows that $Y_n\cap B(z_n,R+3\varepsilon)\neq \emptyset$ and thus
\begin{equation}\label{distlessthanR}
\dist(z_n,Y_n)< R+3\varepsilon
\end{equation}
for all sufficiently large $n\in \N$. In particular, if $R=0$, then $\dist(z_n,Y_n)\to 0$ as $n\to \infty$. Now, suppose that $R>0$ and $0<\varepsilon<R/2$. We have $Y\cap B(z,R)=\emptyset$ and hence, for all sufficiently large $n\in \N$,
$$Y_n\cap B(z_n,R-2\varepsilon)\subset  Y_n\cap B(z,R-\varepsilon) \subset N_\varepsilon (Y\cap B(z,R))=\emptyset.$$
Therefore, $\dist(z_n,Y_n)\geq R-2\varepsilon$ for all large enough $n\in \N$. This in combination with \eqref{distlessthanR} implies that $\dist(z_n,Y_n)\to \dist(z,Y)$ as $n\to \infty$.
\end{proof}

\begin{proposition}\label{equivalencelimitpoints}
Let $(Z,d)$ be a proper metric space. For $n\in \N$ let $Y_n,Y\subset Z$ be closed sets. The following are equivalent.
\begin{enumerate}[label=\normalfont(\arabic*)]
\item $\{Y_n\}_{n\in \N}$ converges to $Y$ in the local Hausdorff sense.
\item There exists a point $z_0\in Z$ such that for each $R,\varepsilon>0$ there exists $N\in \N$ such that for every $n\geq N$ we have
\begin{align*}
&Y_n\cap B(z_0,R-\varepsilon)\subset N_\varepsilon(Y\cap B(z_0,R))\,\,\, \textrm{and}\,\,\,Y\cap B(z_0,R-\varepsilon)\subset N_\varepsilon(Y_n\cap B(z_0,R)). 
\end{align*}
\item For each $y\in Y$ there exists a sequence $y_n\in Y_n$, $n\in \N$, that converges to $y$ and if $y\in Z$ is a limit point of a sequence $y_n\in Y_n$, $n\in \N$, then $y\in Y$. 
\end{enumerate}  
\end{proposition}

\begin{proof}
Suppose (2) is true. Let $y\in Y$. Let $R>2$ such that $y\in B(z_0,R/2)$. For each $k\in \N$ there exists $N_k\in \N$ such that 
\[y\in Y\cap B(z_0,R-1/k)\subset N_{1/k}(Y_n\cap B(z_0,R))\]
for all $n\geq N_k$. This implies that there exists a sequence $y_n\in Y_n$, $n\in \N$, that converges to $y$. Next, let $y_n\in Y_n$, $n\in \N$, be a sequence that has a subsequence $y_{k_n}$, $n\in \N$, converging to a point $y\in Z$. Let $\varepsilon\in (0,1)$ be arbitrary and fix $R>4$ such that $y\in B(z_0,R/2)$. For all sufficiently large $n\in \N$ we have $d(y_{k_n},y)<\varepsilon$, so $y_{k_n}\in B(z_0,R/2+\varepsilon)\subset B(z_0,R-\varepsilon)$. By assumption, $y_{k_n}\in N_{\varepsilon}(Y\cap B(z_0,R))$ for all sufficiently large $n\in \N$, which implies that $y\in N_{2\varepsilon}(Y\cap B(z_0,R))$.  Since $\varepsilon$ is arbitrary, the fact that $Y$ is closed implies that $y\in Y$. Thus, we have verified (3).

Now suppose that (3) is true and suppose that (1) fails. That is, there exist $z_0\in Z$, $R>0$, and $\varepsilon>0$ such that $Y_n\cap B(z_0,R-\varepsilon)\not\subset N_\varepsilon(Y\cap B(z_0,R))$ for infinitely many $n\in \N$ or $Y\cap B(z_0,R-\varepsilon)\not\subset N_\varepsilon(Y_n\cap B(z_0,R))$ for infinitely many $n\in \N$. In the first case we find a sequence $y_{k_n}\in Y_n\cap B(z_0,R-\varepsilon)$ such that $\dist(y_{k_n}, Y\cap B(z_0,R))\geq \varepsilon$ for all $n\in \N$. Since $Z$ is proper, after passing to a further subsequence, we may assume that $y_{k_n}$ converges to a point $y\in \overline B (z_0,R-\varepsilon)$. By the assumption in (3), we have $y\in Y$. Thus, $y\in Y\cap B(z_0,R)$, a contradiction to the condition $\dist(y_{k_n}, Y\cap B(z_0,R))\geq \varepsilon$. In the second case there exists a point $y\in Y\cap B(z_0,R-\varepsilon)$ such that $\dist(y, Y_n\cap B(z_0,R))\geq \varepsilon$ for infinitely many $n\in \N$. By the assumption in (3) there exists a sequence $y_n\in Y_n$ converging to $y$. In particular, since $y\in B(z_0,R-\varepsilon)$ we have $y_n\in Y_n\cap B(z_0,R)$ for all sufficiently large $n\in \N$. This contradicts the condition that $\dist(y,Y_n\cap B(z_0,R))\geq \varepsilon$ for infinitely many $n\in \N$. So, we conclude that (1) is true.
\end{proof}

Hausdorff convergence implies local Hausdorff convergence but the converse is not always true. However, if $Z$ is a compact metric space the two notions agree, as a consequence of part (3) of the above proposition.
\begin{corollary}
If $(Z,d)$ is a compact metric space, then local Hausdorff convergence and Hausdorff convergence of subsets of $Z$ are equivalent.
\end{corollary}

We show that in proper metric spaces local Hausdorff limits exist.

\begin{lemma}\label{existencelimits}
Let $(Z,d)$ be a proper metric space. For $n\in \N$ let $Y_n\subset Z$ be closed sets. Then there exists a subsequence $\{Y_{k_n}\}_{n\in \N}$ that converges to a closed set $Y\subset Z$ in the local Hausdorff sense.
\end{lemma}
\begin{proof}
Let $z_0\in Z$. By Blaschke's theorem \cite[Theorem 7.3.8]{Bur} and a diagonal argument, there exists a subsequence $\{Y_{k_n}\}_{n\in\N}$ such that for each $M\in \N$ there exists a compact set $Y^M\subset \overline{B}(z_0,M)$ with the property that $Y_{k_n}\cap \overline B(z_0,M)$ converges to $Y^M$ in the Hausdorff sense. It follows that $Y^M\subset Y^{M+1}$ for each $M\in \N$ and also $Y^M\cap B(z_0,M)=Y^{N}\cap B(z_0,M)$ for all $N\geq M$. This shows that the set $Y=\bigcup_{M=1}^\infty Y^M$ is a closed subset of $Z$.  

Next, we show that the sequence $Y_{k_n}$ converges to $Y$ in the local Hausdorff sense. Let $R,\varepsilon>0$. Let $M\in \N$, $M>R$. We have 
$$Y_{k_n}\cap B(z_0,R-\varepsilon) \subset Y_{k_n}\cap \overline B(z_0,M) \subset N_\varepsilon(Y^M)\subset N_\varepsilon(Y)$$
for all sufficiently large $n\in \N$. This implies that
\begin{align}\label{inclu1}
Y_{k_n}\cap B(z_0,R-\varepsilon)\subset N_\varepsilon(Y)\cap B(z_0,R-\varepsilon).
\end{align}
Now, we show that
\begin{align}\label{proofofneigh}
N_\varepsilon(Y)\cap B(z_0,R-\varepsilon)\subset N_\varepsilon(Y\cap B(z_0,R)).
\end{align}
Let $y\in N_\varepsilon(Y)\cap B(z_0,R-\varepsilon)$. Then  $\dist (y,Y)<\varepsilon$ and $d(y,z_0)<R-\varepsilon$. There is a point $y_0\in Y$ such that $d(y,y_0)<\varepsilon$ and thus $d(y_0,z_0)<R$. So, we have $y_0\in Y\cap B(z_0,R)$ and $\dist(y,Y\cap B(z_0,R))\le d(y,y_0)<\varepsilon$, which proves \eqref{proofofneigh}. By \eqref{inclu1} and \eqref{proofofneigh} we infer that, for all sufficiently large $n\in \N$,
\[Y_{k_n}\cap B(z_0,R-\varepsilon)\subset  N_\varepsilon(Y\cap B(z_0,R)).\]
In addition, we have
$$Y\cap B(z_0,R-\varepsilon)\subset Y\cap B(z_0,M)= Y^M\cap B(z_0,M)\subset N_\varepsilon(Y_{k_n}\cap \overline B(z_0,M))\subset N_{\varepsilon}(Y_{k_n})$$
for all sufficiently large $n\in \N$. This implies that
$$Y\cap B(z_0,R-\varepsilon) \subset N_\varepsilon(Y_{k_n})\cap B(z_0,R-\varepsilon)\subset N_\varepsilon(Y_{k_n}\cap B(z_0,R)).$$
By Proposition \ref{equivalencelimitpoints} (2) we derive that the sequence $Y_{k_n}$ converges to $Y$ in the local Hausdorff sense. (Note that the subsequence $Y_{k_n}$ depends on the point $z_0$ so we cannot argue that $Y_{k_n}$ converges to $Y$ directly by using the definition of local Hausdorff convergence.)
\end{proof}

\begin{theorem}\label{theorem:equivalence}
Let $(Z,d)$ be a proper metric space. For $n\in \N$ let $Y_n\subset Z$ be closed sets all of which contain a point $z_0\in Z$. Then the following statements are true.

\begin{enumerate}[label=\normalfont(\arabic*)]
\item Let $Y\subset Z$ be a closed set such that $z_0\in Y$. Suppose that $\{Y_n\}_{n\in \N}$ converges to $Y$ in the local Hausdorff sense. Then the sequence $\{(Y_n,d,z_0)\}_{n\in\N}$ converges in the Gromov--Hausdorff sense to $(Y,d,z_0)$.
\item If $\{(Y_n,d,z_0)\}_{n\in\N}$ converges in the Gromov--Hausdorff sense to a complete pointed metric space $(X,\rho,q)$, then there is a subsequence $\{(Y_{k_n},d,z_0)\}_{n\in\N}$ and a closed set $Y\subset Z $  with $z_0\in Y$ such that the sequence $\{Y_{k_n}\}_{n\in \N}$ converges in the local Hausdorff  sense to $Y$ and $(X,\rho,q)$ is isometric to $(Y,d,z_0)$.
\end{enumerate}
\end{theorem}

\begin{proof} First, we prove (1). We fix $\varepsilon>0$ and $R>0$. By assumption, there is $N\in \N$ such that for every $n\ge N$, $Y\cap {B}(z_0,R-\varepsilon/2)\subset N_{\varepsilon/2}(Y_n\cap {B}(z_0,R))$ and $Y_n\cap  B(z_0,R)\subset N_{\varepsilon/2}(Y\cap B(z_0, R+\varepsilon/2))\subset N_{\varepsilon/2}(Y)$. Let $n\geq N$. If $x\in Y_n\cap B(z_0,R)$, then there is a point $y\in Y$ such that $d(x,y)<\varepsilon/2$. Let $f\colon Y_n\cap B(z_0,R)\to Y$ be a mapping satisfying $f(z_0)=z_0$ and $d(x,f(x))<\varepsilon/2$ for every $x\in Y_n\cap B(z_0,R)$. For any pair $x,y\in Y_n\cap B(z_0,R)$, we have
\[ |d(x,y)-d(f(x),f(y))|\le d(x,f(x))+d(y,f(y))<\varepsilon.\]
Finally, let $y\in Y\cap B(z_0,R-\varepsilon)$. Since $Y\cap  B(z_0,R-\varepsilon/2)\subset N_{\varepsilon/2} (Y_n\cap B(z_0,R))$, there is a point $x\in Y_n\cap  B(z_0,R)$ such that $d(y,x)<\varepsilon/2$. This in combination with $d(f(x),x)<\varepsilon/2$ gives $d(f(x),y)<\varepsilon$. Therefore, $y\in N_{\varepsilon}(f(Y_n\cap B(z_0,R)))$. This completes the proof of (1).

Now we prove (2). By Lemma \ref{existencelimits}, there exists a subsequence $\{Y_{k_n}\}_{n\in \N}$ that converges to a closed set $Y\subset Z$ in the local Hausdorff sense. Note that $z_0\in Y$. By (1), the sequence $(Y_{k_n},d,z_0)$ converges to $(Y,d,z_0)$ in the Gromov--Hausdorff sense. Finally, \cite[Theorem 8.1.7]{Bur} implies that $(Y,d,z_0)$ is isometric to $(X,\rho,q)$.
\end{proof}

\begin{definition} Let $X,Y$ be two closed sets in $\C$. We say that $Y$ is a \textit{weak tangent} of $X$ if there is a sequence of points $p_n\in X$ and scales $\lambda_n>0$, $n\in \N$, with  $\lim_{n\to \infty}\lambda_n=0$ such that if 
\[X_n=\frac{X-p_n}{\lambda_n}\]
then the sequence $\{X_n\}_{n\in\N}$ converges in the local Hausdorff sense to $Y$. 
\end{definition}

We remark that weak tangents of $X$ are defined in the literature as metric spaces obtained by considering pointed Gromov--Hausdorff limits of the sequence $(X_n,|\cdot|, 0)$. See, for example, \cite[Definition 8.2.2]{Bur} and \cite[Definition 5.2]{Kin}. However, by Theorem \ref{theorem:equivalence}, weak tangents of $X$ are isometric to local Hausdorff limits of $X_n$, which are subsets of $\C$. Hence, when we refer to weak tangents of a planar set, we will always assume that they are subsets of the plane and that they arise as local Hausdorff limits of a sequence $X_n$ as above.

\subsection{D-components}

Let $D$ be a domain in $\C$ and $Y$ be a weak tangent of $\partial D$. Then 
there is a sequence of points $p_n\in \partial D$ and scales $\lambda_n>0$, $n\in\N$,  with  $\lim_{n\to \infty}\lambda_n=0$ such that
\[\left\{\frac{\partial D-p_n}{\lambda_n}\right\}_{n\in\N}\]
converges in the local Hausdorff sense to $Y$. For $n\in \N$ we set 
\[D_n=\frac{D-p_n}{\lambda_n},\quad  \text{and thus}\quad \partial D_n=\frac{\partial D-p_n}{\lambda_n}.\]
Observe that, for all $n\in \N$, $\diam \partial D_n=(\diam \partial D)/{\lambda_n}$. So, if $\diam \partial D>0$, then
\begin{equation}\label{diaminfinity}
\lim_{n\to\infty} \diam \partial D_n=\infty.
\end{equation}

Let $U$ be a connected component of $\C \setminus Y$. We say that $U$ is a \textit{$D$-component} if there exists $z\in U$ such that $z\in D_n$ for infinitely many $n\in \N$.  

Let $\{\Omega_n\}_{n\in \N}$ be a sequence of domains in $\C$ and let $z_0\in {\C}$ be a point with $z_0\in \Omega_n$ for each $n\in \N$. The \textit{$z_0$-kernel} of $\{\Omega_n\}_{n\in \N}$ is the domain $\Omega$ that is the union of all domains $U$ with the property that $z_0\in U$ and  for each compact set $K\subset U$ there exists $N\in \N$ such that $K\subset \Omega_n$ for all $n\geq N$. Note that the $z_0$-kernel could be the empty set. Moreover, if $\Omega\neq\emptyset$, then $z_0\in \Omega$, $\Omega$ is connected, and for each compact set $K\subset \Omega$ there exists $N\in \N$ such that $K\subset \Omega_n$ for all $n\geq N$. We say that the sequence $\{\Omega_n\}_{n\in \N}$ converges to a domain $\Omega\subset {\C}$ in the \textit{Carath\'eodory sense with base at $z_0$} if $\Omega$ is the $z_0$-kernel of every subsequence of $\{\Omega_n\}_{n\in \N}$ \cite[ \S II.5]{Gol}. 

\begin{lemma}\label{d_componentcompact}
Let $D\subset \C$ be a domain and $Y$ be a weak tangent of $\partial D$. Let $U$ be a $D$-component of $\C\setminus Y$ and $z\in U$. Then there exists a subsequence $\{D_{k_n}\}_{n\in \N}$ of $\{D_n\}_{n\in \N}$  that satisfies the following:
\begin{enumerate}[label=\normalfont(\arabic*)]
\item  For every compact set $K\subset U$ we have $K\subset D_{k_n}$ for all sufficiently large $n\in \N$.
\item  $\{D_{k_n}\}_{n\in \N}$ converges to $U$ in the Carath\'eodory sense with base at $z$.
\end{enumerate}
\end{lemma}

\begin{proof}
First, we prove (1). Let $x\in U$ be a point such that $x\in D_n$ for infinitely many $n\in \N$. Then there exists a subsequence $\{D_{k_n}\}_{n\in \N}$ of $\{D_n\}_{n\in \N}$ such that $x\in D_{k_n}$ for all $n\in \N$. Let $K\subset U$ be a compact set. Then there exists a connected compact set $K'\subset U$ such that $K\cup \{x\}\subset K'$. 

We claim that there exists $N\in \N$ such that $K'\cap \partial D_{k_n}=\emptyset$ for all $n\geq N$. Otherwise, $K'\cap \partial D_{k_n}\neq \emptyset$ for infinitely many $n\in \N$. Therefore, there is a sequence of points $y_n\in K'\cap \partial D_{k_n}$, $n\in \N$, that has a subsequence converging to a point $y\in K'$. By Proposition \ref{equivalencelimitpoints}, $y\in K'\cap Y$. This contradicts the fact that $K'\subset U\subset \C\setminus Y$. Therefore, there exists $N\in \N$ such that $K'\cap \partial D_{k_n}=\emptyset$ for all $n\geq N$. It follows that $K'$ lies in a connected component of $\C\setminus \partial D_{k_n}$. Since $x\in K'\cap D_{k_n}$, we conclude that $K\subset K'\subset D_{k_n}$ for all $n\geq N$.

Now, we prove (2). By (1) we infer that the $z$-kernel of $\{D_{k_n}\}_{n\in \N}$, which we denote by $V$, is non-empty and contains $U$. For the reverse inclusion, we claim that $V\cap Y=\emptyset$. Let $w\in V$. Then there exists $r>0$ such that  $\overline{B}(w,r)\subset V$. By the definition of the $z$-kernel, there is $N_1\in \N$ such that $\overline{B}(w,r)\subset D_{k_n}$ for every $n\ge N_1$ and hence $ \dist(w,\partial D_{k_n})>r$ for every $n\ge N_1$. Hence $w$ cannot be a limit point of $\partial D_{k_n}$. By Proposition \ref{equivalencelimitpoints}, $w\notin Y$, so $V\cap Y=\emptyset$. Since $V$ is connected, we conclude that it is contained in a connected component of ${\C}\setminus Y$. Since $U\subset V$, we derive that $U=V$.

The same argument applies to show that the $z$-kernel of every subsequence of $\{D_{k_n}\}_{n\in \N}$ is $U$, so the conclusion regarding Cath\'eodory convergence follows.
\end{proof}

\section{Main theorem}\label{mainthproof}

A \textit{curve} in a metric space $X$ is a continuous mapping $\gamma$ from a compact interval into $X$. The \textit{trace} of a curve $\gamma$ is its image and it is denoted by $|\gamma|$. We define $\ell(\gamma)$ to be the length of $\gamma$ and $\diam (\gamma)$ to be the diameter of $|\gamma|$. In this section, we prove Theorem \ref{maintheorem}. For this purpose we split the proof into two propositions. 

\begin{proposition}\label{propmain1} Let $D\subset \C$ be a simply connected domain and $A\ge 1$. If $D$ is an $A$-John domain, then for every weak tangent $Y$ of $\partial D$, every $D$-component of $\C \setminus Y$ is an $A$-John domain.
\end{proposition}

\begin{proof}
Let $Y$ be a weak tangent of $\partial D$ and let $U$ be a $D$-component of $\C \setminus Y$. We consider a subsequence $\{D_{k_n}\}_{n\in \N}$ of $\{D_n\}_{n\in \N}$ that satisfies the conclusions of Lemma \ref{d_componentcompact}.  So, if $z,w\in U$, there exists $N\in \N$ such that $z,w\in D_{k_n}$ for every $n\ge N$. For each $n\ge N$, there is a conformal mapping $f_n$ from the unit disk $\D$ onto $D_{k_n}$ such that $f_n(0)=z$ and $f'_n(0)>0$. Lemma \ref{d_componentcompact} (2) implies that $\{D_{k_n}\}_{n\in \N}$ converges to $U$ in the Carath\'eodory sense with base at $z$. Thus, applying Carath\'eodory's kernel convergence theorem \cite[Theorem II.5.1, p.\ 55]{Gol}, we infer that $\{f_n\}_{n\ge N}$ converges locally uniformly in $\D$ to a conformal mapping $f$. Moreover, $f(\D)=U$ and $\{f^{-1}_n\}_{n\ge N}$ converges locally uniformly in $U$ to $f^{-1}$. Therefore, we have
\[\lim_{n\to \infty}f_n^{-1}(w)=f^{-1}(w)\in \D.\]
Let $w_n=f^{-1}_n(w)$ and let $\gamma_n(t)=tw_n$, $0\leq t\leq 1$, be a parametrization of the line segment joining $0$ to $w_n$. We deduce that $\{\gamma_n\}_{n\ge N}$ converges uniformly to $\gamma(t)=tf^{-1}(w)$, $0\leq t\leq 1$, which parametrizes the line segment joining $0$ to $f^{-1}(w)$. Also, $\{f_n\circ \gamma_n\}_{n\ge N}$ converges uniformly to $f\circ \gamma$. We set $\Gamma_n=f_n\circ \gamma_n$ and $\Gamma=f\circ \gamma \colon [0,1]\to U$.

By assumption, $D$ is an $A$-John domain and thus, for each $n\in \N$, $D_n$ is an $A$-John domain as well. By \cite[Theorem 5.2]{Nakki}  the curve $\Gamma_n$, which is actually the hyperbolic geodesic connecting $z$ and $w$ in $D_{k_n}$, is an $A$-John curve for every $n\ge N$. Since $\{f_n'\}_{n\ge N}$ converges locally uniformly in $\D$ to $f'$ and by Lemma \ref{distconvergence}, we derive that, for every $t\in[0,1]$,
\begin{align}
\min \{\ell (\Gamma|_{[0,t]}), \ell(\Gamma|_{[t,1]} ) \}&= \min \{\lim_{n\to \infty} \ell (\Gamma_{n}|_{[0,t]}), \lim_{n\to \infty} \ell (\Gamma_{n}|_{[t,1]})\} \notag \\
&= \lim_{n\to \infty} \min \{\ell (\Gamma_{n}|_{[0,t]}),\ell (\Gamma_{n}|_{[t,1]})\} \notag \\
&\le A\lim_{n\to \infty} \dist(\Gamma_{n}(t),\partial D_{k_n}) \nonumber\\
&=A\dist(\Gamma(t),Y) =A\dist(\Gamma(t),\partial U).\nonumber
\end{align}
So, $\Gamma\colon [0,1]\to U$ is an $A$-John curve joining $z$ to $w$. Since $z,w\in U$ were arbitrary, we infer that $U$
is an $A$-John domain.
\end{proof}

Before proving the converse of Proposition \ref{propmain1} we need some auxiliary lemmas. First, we define some notation. If $\gamma$ is a simple curve connecting two points $a$ and $b$, then for any $x,y\in |\gamma|$ we denote by $\gamma (x,y)$ the subcurve of $\gamma$ joining $x$ to $y$.

\begin{lemma}\label{auxlemma}
Let $D\subset \C$ be a domain and let $\gamma\colon [0,1]\to {D}$ be a simple curve. Suppose there exists $M>0$ and $t_0\in(0,1)$ such that 
\begin{align}\label{assdiam}
\min \{\diam (\gamma|_{[0,t_0]}), \diam(\gamma|_{[t_0,1]} ) \} \ge M \dist (\gamma(t_0),\partial D).
\end{align}
Then there exists a subcurve $\tilde \gamma=\gamma(u,v)$ of $\gamma$ and a point $y\in |\tilde\gamma|$ such that for $r=\dist(y,\partial D)$ we have 
$$N_{r/2}(|\tilde\gamma|)\subset D\quad \text{and}\quad  \min \{\diam \tilde\gamma({u,y}),\diam \tilde\gamma({y,v}) \} \ge (M/3) r.$$
\end{lemma}

\begin{proof} We set $y_0=\gamma(t_0)$ and $d_0=\dist (y_0,\partial D)$. If $B(\gamma(t),d_0/2)\subset D$ for every $t\in [t_0,1]$, then we set $\gamma_1=\gamma|_{[t_0,1]}$, $y_1=\gamma(1)$, and we stop the process. Note that $N_{d_0/2}(|\gamma_1|)\subset D$. Otherwise, let 
\[t_1=\inf \{t\in [t_0,1]: B(\gamma(t),d_0/2) \not\subset  D\}\quad \text{and}\quad \gamma_1=\gamma|_{[t_0,t_1]}.\]
Observe that $N_{d_0/2}(|\gamma_1|)\subset D$. Next, let $y_1=\gamma(t_1)$ and $d_1=\dist (y_1,\partial D)$. Note that $d_1=d_0/2$. If $B(\gamma(t),d_1/2)\subset D$ for every $t\in [t_1,1]$, then we set $\gamma_2=\gamma|_{[t_1,1]}$, $y_2=\gamma(1),$ and we stop the process. Note that $N_{d_1/2}(|\gamma_2|)\subset D$. Otherwise, let  
\[t_2=\inf \{t\in [t_1,1]: B(\gamma(t),d_1/2)\not\subset D\} \quad \text{and}\quad \gamma_2=\gamma|_{[t_1,t_2]}.\]
So, we have $N_{d_1/2}(|\gamma_2|)\subset D$. We proceed inductively and we obtain a partition of $\gamma|_{[t_0,1]}$ into a finite number of non-overlapping subcurves $\gamma_k=\gamma(y_{k-1},y_k)$, $k\in \{1,\dots,n\}$, such that $N_{d_{k-1}/2} (|\gamma_k|)\subset D$, where $d_{k-1}=\dist(y_{k-1},\partial D)=d_0/2^{k-1}$. See Figure \ref{auxlemma:fig} for an illustration. Suppose that $\diam \gamma_k< (M/3) d_{k-1}$ for every $k\in\{1,\dots,n\}$. Then 
\begin{align*}
\diam(\gamma|_{[t_0,1]} )&\le \diam \gamma_1 + \diam \gamma_2 +\dots+\diam \gamma_n\le (M/3)(d_0+d_1+\dots+d_{n-1}) \nonumber\\
&=(M/3)(d_0+d_0/2+ d_0/2^2+\dots+d_0/2^{n-1})\le 2Md_0/3. \nonumber
\end{align*}
However, this contradicts \eqref{assdiam}, which implies that $\diam(\gamma|_{[t_0,1]} )\ge Md_0$. Therefore, there is $k_0\in \{1,\dots,n\}$ such that 
\begin{align}\label{auxlemma:1}
\diam \gamma_{k_0}\ge (M/3)d_{k_0-1}.
\end{align}
Observe that
\begin{align}\label{auxlemma:2}
N_{d_{k_0-1}/2}(|\gamma(y_0,y_{k_0})|) = \bigcup_{k=1}^{k_0} N_{d_{k_0-1}/2}(|\gamma_k|)\subset  \bigcup_{k=1}^{k_0} N_{d_{k-1}/2}(|\gamma_k|) \subset D.
\end{align}

\begin{figure}
	\centering
	\begin{tikzpicture}
		\draw[rounded corners=6pt] (-1.4,2)--(2,2)--(2.5,1)--(3,0.7)--(3.5,1)--(3.7,1.5)--(4,0.8)--(4.5,0.4)--(5,0.7)--(5.5,0.5)--(6,0.2)--(6.5,0.5)--(7.4,0.3)--(8,0)--(7.5,-1)--(6,-0.7)--(5,-1.2)--(4,-1)--(3,-1.5)--(2,-2)--(-1.4,-2) node[above]{$D$};
		
		\draw (-1.4,0) node[left] {$\gamma$}--(6.4,0);
	
		\fill (0,0) circle (1.5pt) node[below] {$y_0$};		
		\draw (0,0) circle (1cm) ;
		\draw[dotted] (0.5,0) circle (1cm) ;
		\draw[dotted] (1,0) circle (1cm) ;
		\draw[dotted] (1.5,0) circle (1cm) ;
		\draw[dotted] (2.2,0) circle (1cm) ;
		
		\fill (2.2,0) circle (1.5pt) node[below] {$y_1$};
		\draw (2.2,0) circle (0.5cm);
		\draw[dotted] (2.7,0) circle (0.5cm);
		\draw[dotted] (3.2,0) circle (0.5cm);
		\draw[dotted] (3.7,0) circle (0.5cm);
		\draw[dotted] (4.25,0) circle (0.5cm);
		
		\fill (4.25,0) circle (1.5pt) node[below, yshift=-0.15cm] {$y_2$};
		\draw (4.25,0) circle (0.25cm);
		\draw[dotted] (4.5,0) circle (0.25cm);
		\draw[dotted] (4.8,0) circle (0.25cm);
		\draw[dotted] (5.1,0) circle (0.25cm);
		\draw[dotted] (5.4,0) circle (0.25cm);
		\draw[dotted] (5.7,0) circle (0.25cm);
		\draw[dotted] (5.9,0) circle (0.25cm);
		
		\fill (5.9,0) circle (1.5pt) node[below,yshift=-0.1cm] {$y_3$};
		\draw (5.9,0) circle (0.125cm);
		\draw[dotted] (6.1,0) circle (0.125cm);
		\draw[dotted] (6.25,0) circle (0.125cm);
		\draw[dotted] (6.4,0) circle (0.125cm);
		
		\fill (6.4,0) circle (1.5pt) node[below,shift={(0.6cm,-0.0cm)}] {$y_4=\gamma(1)$};
	\end{tikzpicture}
	\caption{The partition of $\gamma|_{[t_0,1]}$ into curves $\gamma_1,\gamma_2,\gamma_3,\gamma_4$. The balls $B(y_k,d_k/2)$ are shown in bold.}\label{auxlemma:fig}
\end{figure}

Repeating the same process, we partition $\gamma|_{[0,t_0]}$ into a finite number of subcurves $\gamma_{-l}=\gamma(y_{-l},y_{-l+1})$, $l\in \{1,\dots,m\}$, such that $N_{d_{-l+1}/2} (|\gamma_{-l}|)\subset D$, where $d_{-l+1}=\dist(y_{-l+1},\partial D)$. Thus, there is $l_0\in \{1,\dots,m\}$ such that 
\begin{align}\label{auxlemma:3}
\diam \gamma_{-l_0}\ge (M/3)d_{-l_0+1}\quad \text{and} \quad N_{d_{-l_0+1}/2} (|\gamma(y_{-l_0},y_0) |)\subset D.
\end{align} 

Without loss of generality, assume that $d_{-l_0+1}\ge d_{k_0-1}$. We set $\tilde\gamma=\gamma(y_{-l_0},y_{k_0})$ and $y=y_{k_0-1}$. Then by \eqref{auxlemma:2} and \eqref{auxlemma:3} we have $N_{d_{k_0-1}/2}(|\tilde\gamma|)\subset D$. Finally, by \eqref{auxlemma:1} and \eqref{auxlemma:3}  we infer that
\[\diam \tilde\gamma (y, y_{k_0})=\diam \gamma_{k_0}\ge (M/3)d_{k_0-1}= (M/3)\dist(y,\partial D)\]
and
\[\diam \tilde\gamma ( y_{-l_0},y)\ge \diam \gamma_{-l_0}\ge (M/3) d_{-l_0+1}\ge (M/3) d_{k_0-1}=(M/3)\dist(y,\partial D).\]
The proof is complete.
\end{proof}

The following lemma, known as the separation property for simply connected planar domains, follows from \cite[Lemma 3.3]{Kosk2}. See also \cite[Lemma 2.1]{Kosk} for the current formulation. For the definition of the quasi-hyperbolic geodesic and its fundamental properties, we refer to \cite[Section 8]{Bea}.

\begin{lemma}\label{separation}
Let $D\subset \C$ be a simply connected domain. There exists a universal constant $C\ge 1$ not depending on $D$ such that, for any $x,y\in D$, any quasi-hyperbolic geodesic $\Gamma$ joining $x$ and $y$ in $D$, and every $z\in |\Gamma|$, the ball $B=B(z,C\dist(z,\partial D))$ satisfies $B\cap |\gamma| \neq \emptyset$ for any curve $\gamma$ connecting $x$ to $y$ with $|\gamma| \subset D$.
\end{lemma}

For the proof of our main result we need the following special case of Lemma \ref{separation}.

\begin{corollary}\label{sepcorol}
Let $D\subset \C$ be a simply connected domain with $0\in\partial D$ and let $\Gamma$ be a quasi-hyperbolic geodesic joining two distinct points $x,y$ in $D$. Suppose there is a point $z\in |\Gamma|$ such that $\dist(z,\partial D)= |z|=1$. There exists a universal constant $C_1\geq 2$ not depending on $D$ such that the ball $B(0,C_1)$ intersects any curve connecting $x$ and $y$ in $D$.
\end{corollary}

\begin{proof} By Lemma \ref{separation} there exists a constant $C\ge 1$ such that the ball $B=B(z,C)$ satisfies $B\cap |\gamma| \neq \emptyset$ for any curve $\gamma$ connecting $x$ to $y$ with $|\gamma| \subset D$. Therefore, if $C_1=2C$, the ball $B(0,C_1)$, which contains $B(z,C)$, intersects any curve connecting $x$ and $y$ in $D$.
\end{proof}

\begin{proposition}\label{main2}
Let $D\subset \C$ be a bounded simply connected domain. If, for any weak tangent $Y$ of $\partial D$, every $D$-component of $\C \setminus Y$ is a John domain, then $D$ is a John domain.
\end{proposition}

\begin{proof}
Suppose that $D$ is not a John domain. Then, by \cite[Theorem 2.14]{Nakki}, for every $n\in \N$ there are two distinct points $z_n,w_n\in D$ such that for every curve $\gamma_n\colon [0,1]\to  {D}$ with $\gamma_n(0)=z_n$ and $\gamma_n(1)=w_n$ we have
\begin{align}
\min \{\diam (\gamma_n|_{[0,t]}), \diam(\gamma_n|_{[t,1]} ) \} \ge n\dist (\gamma_n(t),\partial D) \nonumber
\end{align}
for some $t\in(0,1)$. Let $\gamma_n$ be a quasi-hyperbolic geodesic joining $z_n$ to $w_n$ in $D$.

We fix $n\in \N$. By Lemma \ref{auxlemma} there exists a subcurve $\tilde \gamma_n=\gamma_n(u_n,v_n)$ of $\gamma_n$ and a point $y_n\in |\tilde \gamma_n|$ such that for $\lambda_n=\dist(y_n,\partial D)$ we have
\begin{align}\label{diamlemprin}
\min \{\diam \tilde\gamma_n({u_n,y_n}),\diam \tilde\gamma_n({y_n,v_n}) \} \ge (n/3) \lambda_n  \quad \text{and}\quad  N_{\lambda_n/2}(|\tilde\gamma_n|)\subset D.
\end{align}
If $x_n$ is a point on $\partial D$ for which $\dist(y_n,\partial D)=|y_n-x_n|$, then we consider the function
\[f_n(\zeta)=\frac{\zeta-x_n}{\lambda_n},\,\,\zeta\in \C.\]
Let $D_n=f_n(D)$ and thus $\partial D_n=f_n(\partial D)$.
If we set $y'_n=f_n(y_n)$, then 
\begin{align}\label{dist=1}
\dist (y'_n,\partial D_n)=\frac{\dist(y_n,\partial D)}{\lambda_n}=1=|y'_n|.
\end{align}
Moreover, if $\Gamma_n=f_n \circ \tilde\gamma_n$ and $Z_n=f_n(u_n), W_n=f_n(v_n)$, by \eqref{diamlemprin} we have 
\begin{align}\label{diamgamman}
\min\{\diam \Gamma_n({Z_n,y'_n}), \diam \Gamma_n({y'_n,W_n})\}\ge n/3
\end{align}
and
\begin{align}\label{neighborhood}
N_{1/2}(|\Gamma_n|)\subset D_n.
\end{align}

By Lemma \ref{existencelimits} there is a closed set $Y\subset \C$ such that a subsequence of $\{\partial D_n\}_{n\in \N}$ converges in the local Hausdorff sense to $Y$. For the sake of simplicity, we denote the subsequence by $\{\partial D_n\}_{n\in \N}$. By \eqref{diamlemprin} we derive that, for all $n\in \N$,
\[\lambda_n\le \frac{3\diam D}{n}.\]
Since $D$ is a bounded simply connected domain, we infer that $\lim_{n\to\infty} \lambda_n=0$. So, $Y$ is a weak tangent of $\partial D$. Since $0\in \partial D_n$ for all $n\in\N$, by Proposition \ref{equivalencelimitpoints} we infer that $0\in Y$. By passing to a further subsequence, thanks to \eqref{dist=1}, we may assume that $y_{n}'$ converges to a point $y'\in \partial \D$. By Lemma \ref{distconvergence} and  \eqref{dist=1} we have
\[\dist(y',Y)=\lim_{n\to\infty}\dist(y_n',\partial D_n)=1.\]
Thus, $y'$ lies in a connected component $U$ of $\C\setminus Y$. Note that for all sufficiently large $n\in \N$ we have $y'\in B(y_n',1)\subset D_{n}$. This proves that $U$ is a $D$-component. By passing to a further subsequence, and using Lemma \ref{d_componentcompact} (1), we assume that for each compact set $K\subset U$, we have $K\subset D_{n}$ for all sufficiently large $n\in \N$. 

Let $m\in \N$, $m\geq 2$, to be specified later. For every $n>6m$, by \eqref{dist=1} we have $y_n' \in  B(0,m)$ and by \eqref{diamgamman} we deduce that $\Gamma_n({Z_n,y'_n}),\Gamma_n({y'_n,W_n})$ intersect $\partial B(0,m)$ on at least one point, respectively. We denote by $\Gamma^{-}_{n,m}$ a subcurve of $\Gamma_n({Z_n,y'_n})$ that lies in $\br B(0,m)$ and joins $y'_n$ to $\partial B(0,m)$ and we denote by $\Gamma^{+}_{n,m}$ a subcurve of $\Gamma_n({y'_n,W_n})$ that lies in $\br B(0,m)$ and joins $y'_n$ to $\partial B(0,m)$. Therefore, $|\Gamma^{-}_{n,m}|$ and $|\Gamma^{+}_{n,m}|$ are compact and connected and they are contained in $\br B(0,m)\cap D_n$. By Blaschke's theorem \cite[Theorem 7.3.8]{Bur}, there is a subsequence of $|\Gamma_{n,m}^{\pm}|$, ${n>6m}$, that converges in the Hausdorff sense to a compact and connected set $\Gamma_m^{\pm}$. We denote the subsequence by $|\Gamma_{k_n,m}^{\pm}|$, $n\geq 1$, and observe that it depends on $m$. The sets $\Gamma^{+}_m,\Gamma^{-}_m$ are compact and connected, they contain $y'$, they intersect $\partial B(0,m)$ and they are contained in $\br B(0,m)$.

Now, we prove that $\Gamma^-_{m}\cup \Gamma^+_{m}\subset U$. Let $x\in \Gamma^+_{m}$. Then there is a sequence $x_n\in |\Gamma^{+}_{{k_n},m}|$, $n\in \N$, that converges to $x$. By Lemma \ref{distconvergence} we have
\[\lim_{n\to \infty}\dist(x_n,\partial D_{k_n})=\dist (x,Y)\]
and by \eqref{neighborhood} we derive that $\dist(x_n,\partial D_{k_n})\ge 1/2$. Hence, $\dist (x,Y)\ge 1/2$ for every $x\in \Gamma^+_{m}$. The same argument applies to $\Gamma^-_{m}$. Since $\Gamma^-_{m}\cup \Gamma^+_{m}$ is compact, connected, and contains $y'$, we deduce that $\Gamma^-_{m}\cup \Gamma^+_{m}$ lies in  $U$.

Our goal is to show that $U$ is not a John domain. For this purpose we need to prove the following claim. 

\medskip

\noindent
{\bf Claim:} There is a uniform constant $C_1\ge 2$ that does not depend on $m$ such that, if $m>C_1$, then for any $p\in \Gamma^-_{m}\setminus\br B(0,C_1)$ and $q\in \Gamma^+_{m}\setminus\br B(0,C_1)$, any curve connecting $p$ to $q$ in $U$ intersects $B(0,C_1)$. An immediate consequence is that $\partial U \cap \br B(0,C_1)\neq \emptyset$.

To prove this claim, recall that $\Gamma_n$ is a quasi-hyperbolic geodesic between $Z_n$ and $W_n$ in $D_n$ and $y'_n\in |\Gamma_n|$. Corollary \ref{sepcorol} and \eqref{dist=1} imply that there is a uniform constant $C_1 \ge 2$ such that, for every $a_n,b_n\in |\Gamma_n|$ with $y'_n\in|\Gamma_n(a_n,b_n)|$, the ball $B(0,C_1)$ intersects any curve connecting $a_n$ and $b_n$ in $D_n$. 

Now, take $m>C_1$. Suppose there are $p\in \Gamma^-_{m}\setminus \br B(0,C_1)$ and $q\in \Gamma^+_{m}\setminus \br B(0,C_1)$ such that there is a curve $\beta$ joining $p$ to $q$ and lying in $U\setminus  B(0,C_1)$. Then there exists $\varepsilon>0$ such that $B(p,\varepsilon),B(q,\varepsilon)\subset U \setminus\br B(0,C_1)$. Moreover, there are sequences $\{p_n\}$ and $\{q_n\}$ with $p_n\in |\Gamma^{-}_{k_n,m}|$ and $q_n\in |\Gamma^{+}_{k_n,m}|$ such that $p_n\to p$ and $q_n\to q$ as $n\to \infty$. Therefore, there is $N_1\in \N$ such that $p_n\in B(p,\varepsilon)$ and $q_n\in B(q,\varepsilon)$ for every $n\ge N_1$. For each $n\ge N_1$, we take the curve $\beta_n$ which is the concatenation of the line segment joining $p_n$ to $p$, the curve $\beta$ joining $p$ to $q$, and the line segment joining $q_n$ to $q$. By construction, $|\beta_n|\subset U\setminus  B(0,C_1)$ for all $n\ge N_1$. Since $U$ is a $D$-component, there exists $N_2\geq N_1$ such that $|\beta_n|\subset D_{k_n}$ for every $n\ge N_2$. So, $\beta_n$ connects $p_n$ to $q_n$ and $|\beta_n|\subset D_{k_n}\setminus  B(0,C_1)$ for every $n\ge N_2$. This leads to contradiction and thus our claim is true.

Now, we are able to prove that $U$ is not a John domain. On the contrary, assume that $U$ is an $A$-John domain for some $A\ge 1$. Let $m>(2A+1)C_1$, $m\in \N$. Since $m>C_1$, we can find $p\in(\Gamma^-_{m}\cap \partial B(0,m))\setminus \br B(0,C_1)$ and $q\in (\Gamma^+_{m}\cap \partial B(0,m))\setminus \br B(0,C_1)$. By \cite[Theorem 2.14]{Nakki} there is a curve $\gamma\colon [0,1]\to U$ such that $\gamma(0)=p$, $\gamma(1)=q$ and 
\begin{equation}\label{johndiamcond}
\min\{\diam (\gamma|_{[0,t]}), \diam(\gamma|_{[t,1]} )\} \leq A \dist(\gamma(t),\partial U),
\end{equation}
for all $t\in [0,1]$. By our claim we infer that $|\gamma|\cap B(0,C_1)\neq \emptyset$ and $\partial U \cap \br B(0,C_1)\neq \emptyset$. Thus, there is $z\in |\gamma|\cap B(0,C_1)$ such that $\dist (z,\partial U)\le 2C_1$. Without loss of generality, assume that $\diam \gamma(p,z)\le \diam \gamma(z,q)$. Then, by \eqref{johndiamcond} we have
\[\diam \gamma(p,z)\le A \dist(z,\partial U)\le 2AC_1.\]
However, since $m>(2A+1)C_1$, we have $\diam \gamma(p,z)\ge m-C_1>2AC_1$, which leads to contradiction. Hence, $U$ is not a John domain.

Concluding, we have found a weak tangent $Y$ of $\partial D$ for which there is a $D$-component $U$ of $\C\setminus Y$ that is not a John domain. So, the proof is complete.
\end{proof}

Now, Theorem \ref{maintheorem} follows directly from Proposition \ref{propmain1} and Proposition \ref{main2}.

\section{Further properties of weak tangents}\label{propertieswt}

In this section we prove Theorem \ref{sum}, which gives several geometric properties of weak tangents of simply connected John domains that are interesting in their own right. Each part of the theorem is proved in a separate lemma.

\begin{lemma}\label{dcompounbounded}
Let $D\subset \C$ be a simply connected John domain. For every weak tangent $Y$ of $\partial D$, every $D$-component of $\C\setminus Y$ is unbounded and simply connected.
\end{lemma}

The idea of the lemma is that if there exists a bounded $D$-component, then the domain $D$ must have narrow passages, which is prohibited by the John domain condition. The last claim of the lemma about simple connectivity is more general and does not rely on the John property.

\begin{proof}
Suppose that $D$ is an $A$-John domain for some $A\geq 1$. Let $Y$ be a weak tangent of $\partial D$ and let $U$ be a $D$-component of $\C\setminus Y$. By Lemma \ref{d_componentcompact} there exists a subsequence $\{D_{k_n}\}_{n\in \N}$ of $\{D_n\}_{n\in \N}$ such that for every compact set $K\subset U$ we have $K\subset D_{k_n}$ for all sufficiently large $n\in \N$. Let $z\in U$. There is $N_1\in\N$ such that $z\in D_{k_n}$ for all $n\ge N_1$. Let $R>0$. By \eqref{diaminfinity}, $\diam \partial D_n\to \infty$ as $n\to \infty$ and hence there is $N_2\ge N_1$ such that $D_{k_n}\not\subset \br B(z,2R)$ for every $n\ge N_2$. So, for every $n\ge N_2$, there is $w_n\in D_{k_n} \setminus  \br B(z,2R)$. Since $D_{k_n}$ is an $A$-John domain, we derive that, for every $n\ge N_2$, there is an $A$-John curve $\Gamma_n\colon [0,1]\to D_{k_n}$ with $\Gamma_n(0)=z$ and $\Gamma_n(1)=w_n$. For each $n\ge N_2$, let $E_n=\Gamma_n([0,t_n])$ where $t_n=\min\{t:\Gamma_n(t)\in\partial B(z,R)\}$. So, $E_n$ is a compact and connected set that is contained in $\br B(z,R)$ and it contains $z$ and a point on $\partial B(z,R)$. By Blaschke's theorem \cite[Theorem 7.3.8]{Bur}, there is a subsequence $\{E_{m_n}\}_{n\ge 1}$ of $\{E_n\}_{n\ge N_2}$ that converges in the Hausdorff sense to a compact and connected set $E\subset \br B(z,R)$ such that $z\in E$ and $E\cap \partial B(z,R)\neq \emptyset$. We claim that $E\subset U$. Since $E$ is connected and $z\in E\cap U$, it suffices to prove that $\dist(E,Y)>0$.

To show this, we choose $x\in E$ such that $\dist(E,Y)=\dist(x,Y)$. If $x=z$ this is trivial, so we suppose that $x\neq z$. There is a sequence $x_n\in E_{m_n}$, $n\in \N$, that converges to $x$. So, there is $t_n'\in [0,1]$ such that $x_n= \Gamma_{m_n}(t_n')$ for all $n\in \N$. By Lemma \ref{distconvergence} we have
\begin{align}
\dist(x,Y)&=\lim_{n\to\infty}\dist(x_n,\partial D_{k_n})\ge \frac{1}{A}\lim_{n\to\infty}\min\{\ell( \Gamma_{m_n}|_{[0,t_n']}),\ell (\Gamma_{m_n}|_{[t_n',1]})\}\nonumber\\
&\ge \frac{1}{A}\lim_{n\to\infty}\min\{|z-x_n|,|x_n-w_n|\}=\frac{1}{A}|z-x|>0,\nonumber
\end{align}
where we used the fact that $|z-x_n|\le R\le |x_n-w_n|$ for all $n\in \N$. So, the claim has been proved. Since $E\subset U$, we have $\diam U\ge \diam E \ge R$. Letting $R\to \infty$ we get $\diam U =\infty$ and hence $U$ is unbounded.

Finally, we prove that $U$ is simply connected. Suppose that $U$ is not simply connected. Then there is a Jordan region $V$ such that $\partial V\subset U$ and $V$ contains a connected component $C_0$ of $\partial U$. Since $\partial V$ is compact, we have $\partial V\subset D_{k_n}$ for all large enough $n\in\N$. By assumption, $ D_{k_n}$ is simply connected and thus $C_0\subset V\subset \br V \subset D_{k_n}$ for all large enough $n\in\N$. Let $\zeta \in C_0$. Then $\dist (\zeta, \partial D_{k_n})\ge \dist (\zeta,\partial V)\ge \dist(C_0,\partial V)$ for all large enough $n\in\N$. Therefore, applying Lemma \ref{distconvergence}, we deduce that
\[\dist (\zeta, Y)=\lim_{n\to\infty}\dist (\zeta, \partial D_{k_n})\ge \dist(C_0,\partial V)>0.\]
Since $\zeta\in C_0$ was arbitrary, we have $\dist (C_0, Y)>0$. This contradicts the fact that $C_0\subset Y$. So, $U$ is simply connected.
\end{proof}

We prove a general property of weak tangents of the boundary of a bounded simply connected domain.

\begin{lemma}\label{unboundedY} Let $D\subset \C$ be a bounded simply connected domain. For every weak tangent $Y$ of $\partial D$, every connected component of $Y$ is unbounded.
\end{lemma}

\begin{proof} Let $Y$ be a weak tangent of $\partial D$. Suppose that $Y$ has one bounded connected component $Y_0$. First, suppose that $Y$ has another connected component $Y_1$. By Zoretti's theorem \cite[Corollary VI.3.11, p.\ 109]{Zor} there is a Jordan region $V$ such that $\partial V\cap Y=\emptyset$, $Y_0\subset V$, and $Y_1\subset \C\setminus \overline{V}$. Let $z\in Y_0$ and $w\in Y_1$. By Proposition \ref{equivalencelimitpoints} there are sequences $z_n,w_n\in \partial D_n$, $n\in\N$, such that $z_n\to z$ and $w_n\to w$ as $n\to \infty$. So, there is $N\in\N$ such that, for every $n\ge N$, $z_n\in V$ and $w_n\in \C \setminus \br V$. By assumption, $\partial D_n$ is connected for all $n\in\N$. So, for every $n\ge N$ there is  $x_n\in \partial D_n \cap \partial V$. Then there is a subsequence of $\{x_n\}_{n\ge N}$ that converges to a point $x\in \partial V$. Proposition \ref{equivalencelimitpoints} implies that $x\in Y$, which leads to contradiction, because $\partial V\cap Y=\emptyset$.

Now, suppose that $Y=Y_0$. Since $Y$ is bounded, there is $R>0$ such that $Y\subset B(0,R)$. By \eqref{diaminfinity} there is $N_1\in\N$ such that, for every $n\ge N_1$, $\diam \partial D_n >2R$ and hence $\partial D_n\not\subset \br B(0,R)$. So, for each $n\ge N_1$, there exists $z_n\in \partial D_n\setminus \br B(0,R)$. Now, let $y\in Y$. By Proposition \ref{equivalencelimitpoints} there is a sequence $y_n\in \partial D_n$, $n\in\N$, that converges to $y$. So, there is $N_2\ge N_1$ such that $y_n \in B(0,R)$ for all $n\ge N_2$. Since $\partial D_n$ is connected, for every $n\ge N_2$ there is $x_n\in \partial D_n \cap \partial B(0,R)$. By compactness there is a subsequence of $\{x_n\}_{n\ge N_2}$ converging to a point $x\in \partial B(0,R)$. Also, by Proposition  \ref{equivalencelimitpoints}, $x\in Y$, which contradicts the fact that $Y\cap \partial B(0,R)=\emptyset$.
\end{proof}
 
The following lemma, which is an application of our main  theorem, gives a uniform bound for the number of $D$-components in the case that $D$ is a simply connected  John domain.

\begin{lemma}\label{upperbound}
Let $A\ge 1$. Let $D\subset \C$ be a simply connected $A$-John domain. For every weak tangent $Y$ of $\partial D$, the number of $D$-components of $\C\setminus Y$ is bounded by $4\pi A$.
\end{lemma}

\begin{proof} Let $Y$ be a weak tangent of $\partial D$. Suppose that there is at least one $D$-component of $\C\setminus Y$. Let $U_1,U_2,\dots,U_N$ be distinct $D$-components of $\C\setminus Y$, where $N\ge 1$. By Proposition \ref{propmain1} these connected components are $A$-John domains and by Lemma \ref{dcompounbounded} they are unbounded. Thus, there is $R>0$ such that $\partial B(0,R)$ intersects $U_1,U_2,\dots,U_N$ and, for each $i=1,2,\dots,N$, there exist $z_i\in U_i\cap B(0,R)$ and $w_i\in U_i\setminus \br B(0,3R)$. Then there is a curve $\gamma_i\colon [0,1]\to U_i$ such that $\gamma_i(0)=z_i$, $\gamma_i(1)=w_i$ and for every $t\in [0,1]$,
\[\min\{\ell (\gamma_i|_{[0,t]}),\ell (\gamma_i|_{[t,1]})\}\le A\dist (\gamma_i(t),\partial U_i).\]
We choose $t_i\in (0,1)$ such that $\gamma_i(t_i)\in \partial B(0,2R)$. So, we have
\[R\le A\dist (\gamma_i(t_i),\partial U_i)\le 2AR\theta^i(2R),\]
where $2R\theta^i(2R)$ is the length of the smallest arc on $\partial B(0,2R)$ joining $\gamma_i(t_i)$ to $\partial U_i$. Now, we take the sum for all $i=1,2,\dots,N$ and we get
\[NR\le  2AR\sum_{i=1}^{N}\theta^i(2R)  \le 4AR\pi.\]
So, we derive that $N\le 4\pi A$.
\end{proof}

Based on the above lemmas, we establish a result that describes a weak tangent $Y$ of $\partial D$ in terms of the boundaries of $D$-components.

\begin{lemma}\label{components}
Let $D\subset \C$ be a bounded simply connected John domain. For every weak tangent $Y$ of $\partial D$ and for every $z\in Y$ there exists a $D$-component $U$ such that $z\in \partial U$. In particular, $Y$ is equal to the union of the boundaries of $D$-components of $\C\setminus Y$.
\end{lemma}

\begin{proof} Suppose that $D$ is an $A$-John domain for some $A\geq 1$. Let $Y$ be a weak tangent of $\partial D$. Let $z\in Y$ and $\varepsilon>0$. By Lemma \ref{unboundedY}, $Y$ is unbounded and thus there is $w\in Y \setminus   B(z,2\varepsilon)$. By Proposition \ref{equivalencelimitpoints}, there are  sequences $z_n,w_n\in \partial D_n$, $n\in\N$, such that $z_n\to z$ and $w_n\to w$ as $n\to \infty$. So, there is $N\in \N$ such that, for every $n\ge N$, $z_n\in B(z,\varepsilon)$ and $w_n\in B(w,\varepsilon)$. For each $n\ge N$, by \cite[Theorem 6.5]{Nakki} there is a curve $\gamma_n\colon [0,1]\to  \br {D}_n$ with $\gamma_n(0)=z_n$, $\gamma_n(1)=w_n$, $\gamma_n((0,1))\subset D_n$ and
\[\min \{\ell (\gamma_n|_{[0,t]}), \ell(\gamma_n|_{[t,1]} ) \}\le A\dist(\gamma_n(t),\partial D_n),\]
for every $t\in [0,1]$. Take $t_n\in(0,1)$ such that $\zeta_n:=\gamma_n(t_n)\in \partial B(z,\varepsilon)\cap D_n$. Then 
\[\min\{|z_n-\zeta_n|,|\zeta_n-w_n|\}\le A\dist(\zeta_n,\partial D_n). \]
Moreover, there is a subsequence $\{\zeta_{k_n}\}_{n\in \N}$ of $\{\zeta_n\}_{n\ge N}$ that converges to a point $\zeta({\varepsilon})\in \partial B(z,\varepsilon)$. So, taking limits in the above relation as $n\to \infty$, by Lemma \ref{distconvergence} we infer that
\begin{equation}\label{distgreaterepsilon}
\varepsilon=|z-\zeta({\varepsilon})|\le A\dist(\zeta({\varepsilon}),Y).
\end{equation} 
This implies that $\zeta({\varepsilon})\notin Y$. So, $\zeta({\varepsilon})$ belongs to a connected component $U_{\varepsilon}$ of $\C\setminus Y$. 

We now show that $U_{\varepsilon}$ is a $D$-component. Lemma \ref{distconvergence} and \eqref{distgreaterepsilon} give
\[\lim_{n\to\infty}\dist(\zeta(\varepsilon),\partial D_{k_n})=\dist (\zeta({\varepsilon}),Y)\ge \varepsilon/A\]
and thus $\overline B(\zeta(\varepsilon),\varepsilon/(2A))\cap \partial D_{k_n}=\emptyset$ for all sufficiently large $n\in \N$. Moreover, $\zeta_{k_n}\in  B(\zeta(\varepsilon),\varepsilon/(2A))\cap D_{k_n}$ for all sufficiently large $n\in \N$. Hence, $\overline B(\zeta(\varepsilon),\varepsilon/(2A)) \subset D_{k_n}$ and thus $\zeta(\varepsilon)\in D_{k_n}$ for all sufficiently large $n\in \N$. This proves that $U_\varepsilon$ is a $D$-component.

We apply the above argument for $\varepsilon=1/n$, $n\in\N$, and we get a sequence of points $\zeta({1/n})$, $n\in\N$, such that $|z-\zeta(1/n)|=1/n$ and $\zeta(1/n)\in U_{1/n}$, where $U_{1/n}$ is a $D$-component of $\C\setminus Y$. By Lemma \ref{upperbound} the number of $D$-components is uniformly bounded and thus we derive that $\zeta(1/n)$  belongs to a $D$-component $U_{0}$ for infinitely many $n\in \N$. Since $\zeta(1/n)\to z$ as $n\to\infty$ and $z\in Y\subset \C\setminus U_0$, we have $z\in \partial U_{0}$. 
\end{proof}

The following lemma is independent of all our previous results and asserts that the boundary of a simply connected John domain has a uniformly bounded number of connected components. The idea is that if this number is not bounded, then the domain must have narrow passages or bottlenecks, which is prohibited by the John domain condition.

\begin{lemma}\label{johnboundcomp} Let $D\subset \C$ be a simply connected $A$-John domain for some $A\ge 1$. Then the number of connected components of $\partial D$ is bounded by $4\pi A$.
\end{lemma}

\begin{proof}
If $\partial D$ is connected then the conclusion is immediate. So, assume that $\partial D$ has at least two connected components. Let $C_1,C_2,\dots,C_N$ be distinct connected components of $\partial D$, where $N\geq 2$. Note that since $D$ is simply connected, $D$ and $C_1,C_2,\dots,C_N$ are all unbounded. We take 
$$d=\min \{\dist(C_i,C_j):i,j\in\{1,2,\dots,N\}\,\,\, {\rm and}\,\,\,i\neq j\}$$
and $\varepsilon=d/3$. So, $N_{\varepsilon} (C_i)\cap N_\varepsilon(C_j)=\emptyset$ for $i\neq j$.  There is $R>0$ such that $\partial B(0,R)$ intersects $C_1,C_2,\dots,C_N$ and $D\cap B(0,R)\neq \emptyset$. Let $z\in D\cap B(0,R)$. For each $i=1,2,\dots,N$ we take $z_i\in (D\cap N_{\varepsilon} (C_i))\setminus \br B(0,3R)$. Then there is a curve $\gamma_i\colon [0,1]\to D$ such that $\gamma_i(0)=z$, $\gamma_i(1)=z_i$ and for every $t\in [0,1]$,
\[\min\{\ell (\gamma_i|_{[0,t]}),\ell (\gamma_i|_{[t,1]})\}\le A\dist (\gamma_i(t),\partial D).\]
We choose $t_i\in (0,1)$ such that $\gamma_i(t_i)\in \partial B(0,2R)$. So, we have
\[R\le A\dist (\gamma_i(t_i),C_i)\le 2AR\theta^i(2R),\]
where $2R\theta^i(2R)$ is the length of the smallest arc on $\partial B(0,2R)$ joining $\gamma_i(t_i)$ to $C_i$. Taking the sum for $i=1,2,\dots,N$, we get $NR\le 4AR\pi$ and thus $N\le 4\pi A$.
\end{proof}

Our last result, which is a consequence of our main theorem and the previous lemmas, gives an alternative proof of Theorem 5.6 in \cite{Kin}.

\begin{corollary}\label{corol}
Let $A\ge 1$. Let $D\subset\C$ be a bounded simply connected $A$-John domain. Then the number of connected components of every weak tangent $Y$ of $\partial D$ is bounded by $16 \pi^2A^2$.
\end{corollary}

\begin{proof}
Let $Y_1,Y_2,\dots,Y_N$ be distinct connected components of $Y$, where $N\ge 1$. By Lemma \ref{components} every connected component of $Y$ contains a connected component of the boundary of a $D$-component of $\C\setminus Y$. Thus there is a mapping from the set $B=\{Y_1,Y_2,\dots,Y_N\}$ into the set $E$ of $D$-components of $\C\setminus Y$. By Lemma \ref{dcompounbounded} and Theorem \ref{maintheorem}, every $D$-component of $\C\setminus Y$ is a simply connected $A$-John domain. So, Lemma \ref{johnboundcomp} implies that the boundary of each $D$-component has at most $4\pi A$ connected components. We conclude that each element of $E$ has at most $4\pi A$ pre-images in $B$. By Lemma \ref{upperbound}, $E$ has at most $4\pi A$ elements. Since $N$ is equal to the sum of the numbers of pre-images of the elements in $E$, we derive that $N\le (4\pi A)^2$.  
\end{proof}

\end{document}